\documentclass[12pt]{amsart}
\usepackage{url}
\usepackage{hyperref}
\usepackage{amsrefs}
\usepackage{latexsym}
\usepackage{tikz}
\usepackage[nocenter]{qtree}
\usepackage{breqn}
\usepackage[edges]{forest}
\usepackage{amsmath,amscd,amssymb}
\usepackage{graphicx}
\usepackage{caption}
\usepackage{subcaption}
\usepackage{float}
\restylefloat{table}
\usepackage{amsmath,amscd,amssymb}
\usepackage{caption}
\usepackage{xcolor}
\usepackage{graphicx,wrapfig,lipsum}
\usepackage{diagbox}
 2

\newtheorem{theorem}{Theorem}[section]
\newtheorem{prop}{Proposition}[section]

\newtheorem*{theorem*}{Theorem}
\newtheorem{definition}[theorem]{Definition}

\newtheorem{remark}[theorem]{Remark}

\numberwithin{equation}{section}

\date{}

\begin{document}

\title[Complete maximal maps and their singularities]{On the Complete maximal maps and
their singularities}

\author[R. Bardhan]{Rivu Bardhan}
\address{Department of Mathematics, Shiv Nadar University, Dadri 201314, Uttar Pradesh, India.}
\email{rb212@snu.edu.in}

\author[I. Biswas]{Indranil Biswas}
\address{Department of Mathematics, Shiv Nadar University, Dadri 201314, Uttar Pradesh, India.}
\email{indranil.biswas@snu.edu.in, inrdranil29@gmail.com}

\author[P. Kumar]{Pradip Kumar}
\address{Department of Mathematics, Shiv Nadar University, Dadri 201314, Uttar Pradesh, India.}
\email{pradip.kumar@snu.edu.in}

\author[S. Paul]{Subham Paul}
\address{Department of Mathematics, Indian Institute of Technology Patna, Bihta 801106, Bihar, India.}
\email{subham\_2021ma25@iitp.ac.in}

\date{}



\begin{abstract}
This article investigates the global structure of maximal surfaces (space-like immersions with zero mean 
curvature) in Lorentz-Minkowski $3$-space \(\mathbb{E}^3_1\), especially focusing on the interplay between 
genus, the number 
of singular components—loci, and simple ends. We construct complete maximal maps with arbitrarily many 
singular components for any genus \(p\,\geq\, 0\) with simple ends.
\end{abstract}
\maketitle

\section{Introduction}

The maximal surfaces are defined to be the space-like 
immersions, with zero mean curvature, in the Lorentz Minkowski space $\mathbb{E}_1^3$;
note that this is similar to the definition of minimal surfaces in $\mathbb{R}^3$. These surfaces emerge as 
solution to the problem of locally maximizing the area among the space-like surfaces in $\mathbb{E}_1^3$.

While the maximal surfaces possess some similarities with the minimal surfaces, significant differences
emerge when their global properties are compared. For instance, various examples of complete minimal surfaces
in $\mathbb{R}^3$ exist 
(e.g., the Catenoid, Enneper surface, Costa surfaces, etc.). In contrast, the only complete maximal surface
is the plane.

Thus, to ensure the existence of global non-planar complete maximal surfaces, singularities must be allowed. 
We call a maximal surface with singularities a maximal map. Analogous to minimal immersions in 
\(\mathbb{R}^3\), the maximal map also admits a Weierstrass-Enneper representation \cite{Estudillo1992}, 
\cite{KOBAYASHI1983}. Let \(M\) be a Riemann surface, and let \(g\) be a meromorphic function on
$M$ and \(\omega\) a 
holomorphic one-form on $M$, which together form the Weierstrass data for a maximal map. Then the
singular set of the maximal map is a subset of
 \[
 \{\,p \,\in\, M\ \, :\ \  |g(p)| \,=\, 1\, \;\text{ or }\, \; \omega(p) \,=\, 0\}.
\]

The subset \(\{p \,\in\, M\,\,\big\vert\,\,\,\, |g(p)| \,=\, 1\}\, \subset\, M\) is referred to as the
non-degenerate singular locus. Table \ref{Table of singular components} sheds light on the number of
connected components of the non-degenerate singular locus (referred to as singular components) for various
complete maximal maps.

\begin{table}[ht]
 \centering
\begin{tabular}{|m{1cm}|m{4cm}|m{3cm}|m{2cm}|m{2cm}|m{1cm}|}
\hline
       S.N. &Example of Maximal map & No. of singular components & Genus & Embedded in wider sense & No. of ends\\ \hline
       1 & Lorentzian catenoid & 1 & 0 & Yes & 2\\ \hline 
        2 &Lorentzian helicoid & 1 & 0 & Yes & 2 \\ \hline
         3& Lorentzian Enneper surface & 1 & 0 & No & 1 \\ \hline
        4& Lorentzian Jorge-Meek's surface & 1 & 0 & No & $n$\\ \hline
         5& Maximal map constructed in \cite{SaiPradip0genus} for a given singular curve $C$ & No. of components of $C$ & 0 & No & $n$\\ \hline
        6& Lorentzian Costa Surface (\cite{chen2024singularitiesmaxfacesconstructednodeopening}) & 3 & 1 & Yes & 3\\ \hline
         
        7& Lorentzian Costa-Hoffman-Meeks (CHM) surface (\cite{chen2024singularitiesmaxfacesconstructednodeopening}) ($n_2=m>2$) & $m+1$ & $m-1$ & Yes & 3\\ \hline
       8&  Generalized CHM surface ($n_{L-1}=2L-3$, $L>3$) & $1+ (L-2)(2L-3)$ & $2(L-2)^2$ & Yes & L\\ \hline
    9&Higher Genus Lorentzian Chen-Gackstatter Surfaces \cite{bardhan2023higher} & $p+1$ & $p$ & Yes & $1$\\ \hline
    \end{tabular}
    \caption{Summary of some known complete maximal maps, illustrating the relationship between
the number of singular components and the genus and ends.}
\label{Table of singular components}
\end{table}

From entry 5 of Table~\ref{Table of singular components}, we see that there is no upper bound on the number of 
singular components for a genus-zero complete maximal map (one can increase the number of singular components 
arbitrarily by choosing the singular curve $C$ to have more components in the construction done in
\cite{SaiPradip0genus}).

It was not known, however, whether a similar phenomenon could occur for higher-genus maximal maps. In fact, all known examples to date (see Table~\ref{Table of singular components}) satisfy the heuristic lower bound that a maximal map of genus $p$ has at least $p+1$ singular components.  These observations lead to the following question:
\medskip

\noindent
\textit{ For a given genus $p\,\geq\, 0$, can one construct a complete maximal map that has 
  \begin{enumerate}
  \item  an arbitrarily large number of singular components, and 
  \item all ends are complete?
  \end{enumerate}}
\medskip

For genus \(p \,=\, 0\), the existence of complete maximal maps with arbitrarily many non-degenerate singular 
components and simple ends is already established in \cite{SaiPradip}. Our focus here is on the case \(p \,>\, 
0\), where we seek to maximize the freedom in configuring the non-degenerate singular locus. Crucially, we 
relax regularity constraints by allowing branch singularities (where the differential \(\omega\) vanishes) in 
addition to non-degenerate singularities (\(|g|\, =\, 1\)). This expanded framework --- working with maximal maps 
rather than maxfaces (which permit only non-degenerate singularities) --- unlocks greater topological 
flexibility.  By embracing both singularity types, we systematically explore the full parameter space for 
maximal maps of higher genus.

In Section~\ref{sec:Mainsec}, for given $n\,>\,0$ and $p\,>\,0$, we explore the relationship between $n$
and $p$, such that we have complete maximal maps of genus $p$ whose non-degenerate singular loci have $n$
singular loops.  More precisely main results of this article are as follows:

\begin{enumerate}
\item For given \(p\) and \(N \,>\, p+2\), there exists a complete maximal map of genus \(p\) with \(2N\) complete
ends and at least \(2N\) singular components. (See Theorem \ref{thm:genus,singularity}.)

\item Given \(p\),\, \(m\) and \(N\) such that \(N - \lceil\frac{m}{2}\rceil \,>\, p+2\), there exists a complete
maximal map of genus \(p\) with \(2N\) complete ends, \(2m\) simple ends, and at least \(2N\) singular components.
(See Theorem \ref{thm:simple and Complete ends}.)
\end{enumerate} 

We use the Weierstrass–Enneper representation to explicitly build such maximal maps and solve the 
corresponding period problem (see \eqref{period condition}).  To solve the period problem, various methods are 
available, which are modifications of methods used in the context of minimal surfaces. Below we briefly recall
these methods and explain the challenges while adopting these in the case of maximal maps.

Weber and Wolf in \cite{weber2002teichmuller} introduced orthodisks and, using calculus in Teichm\"uller space, 
proved the existence of reflexive orthodisks of any genus. Each reflexive orthodisk of genus \(p\) gives rise to a 
minimal immersion of genus \(p\). For maximal maps, this method (along with some modifications) was effective in 
generating complete maximal maps of genus \(p\) \cite{bardhan2023higher}. However, it produces at most \(p+1\) 
singular components for genus \(p\) maximal maps.

Another method proposed by Traizet, \cite{traizet2002embedded}, for the construction of minimal surfaces can also be 
modified to construct maximal maps \cite{chen2024singularitiesmaxfacesconstructednodeopening}. 
However, this method also results in maximal maps of genus \(p\) with \(p+1\) singular components. There are a few 
other methods, such as using symmetry to solve the period problem, which were explored by Meeks 
\cite{Meeks1989} and followed up, in the case of maximal maps, by the authors of \cite{Kim2006}, 
\cite{Fujimori2009}.    

Apart from Traizet's strategy in \cite{traizet2002embedded} for removing nodes, in other methods, many times, we get 
only a maximal map (with branch singularities) and not the maxface.  Further, all of these approaches primarily 
focus on the construction of embedded (at least in a wider sense) surfaces, and thus, using these methods, we 
cannot expect to generate surfaces with a large number of singular components.

With this in mind, to address the problem of constructing a maximal map of genus \(p\) with complete and/or 
embedded ends and any number of singular components, we follow a very classical approach, which is the
content of Section \ref{sec:Mainsec}. The complete construction is presented in Section \ref{sec:Mainsec}, 

The article is organized as follows: In Section \ref{sec:prelim} we recall some basic facts about maximal maps.  In Section \ref{sec:Mainsec} we prove Theorems 
\ref{thm:genus,singularity} and \ref{thm:simple and Complete ends}.

\section{Preliminaries}\label{sec:prelim}

By $\mathbb{E}_1^3$ we denote the real vector space $\mathbb{R}^3$ equipped with the Lorentzian metric 
$\mathrm{d}x^2+\mathrm{d}y^2-\mathrm{d}z^2$. We begin by recalling the Weierstrass–Enneper representation for 
maximal maps.

\smallskip 

\noindent \textbf{Weierstrass–Enneper representation of a maximal
map} (cf. \cite{KOBAYASHI1983}, \cite{Estudillo1992}): \; Let $M$ be a connected Riemann surface, $g$ a
meromorphic function on $M$, and $\omega$ a holomorphic 1-form on $M$, such that both $g\,\omega$ and
$g^{-1}\omega$ are actually holomorphic 1-forms on $M$. These data $(M,\,g,\,\omega)$ are the Weierstrass
data for a maximal map; we assume that they satisfy the following two conditions:
\begin{enumerate}
\item $|g(p)|\,\not\equiv\, 1$ on $M$, 
\item For every closed loop $\gamma$ in $M$, 
\begin{equation}\label{period condition}
\int_{\gamma} g\,\omega + \overline{\int_{\gamma} g^{-1}\omega} \;=\; 0, \qquad \Re\! \int_{\gamma} \omega \;=\; 0
\end{equation}
(which are the period conditions).
\end{enumerate} 
Given such data, one constructs a map $X\,:\, M\, \longrightarrow\, \mathbb{E}_1^3$ by 
\begin{equation}\label{maximal_map}
X(z)\ \;=\ \; \mathrm{Re} \int_{z_0}^{\,z} \Big(\frac{1}{2}(g^{-1}+g),\; \frac{i}{2}(g^{-1}-g),\; 1\Big)\,\omega 
\end{equation}
for some base point $z_0\,\in\, M$. This $X$ is a (branched) maximal immersion $M \,\longrightarrow\,
\mathbb{E}_1^3$; conversely, any (branched) maximal map can be represented in this form. The pair
$(g,\,\omega)$ is called the \emph{Weierstrass data} of the maximal map $X$.

The first fundamental form (induced metric) on $M$ for the immersion $X$ is 
\begin{equation}\label{ds}
ds^2 \;=\; \frac{1}{4}\left(|g|^{-1}-|g|\right)^2 |\omega|^2~.
\end{equation}
The \emph{singular set} of the maximal map $X$ (with data $(M,\,g,\,\omega)$) is 
\[
\{p\,\in\, M\,\, \big\vert\,\,\, |g(p)|\,=\,1\,\,  \text{ or }\,\, \omega(p)\,=\,0\}\ \subset\ M.
\] 
We focus on the \emph{non-degenerate} case, so we define 
\begin{equation}\label{eq:singset}
\Sigma\ \;:=\ \; \{\,p\,\in\, M\,\, \big\vert\,\,\, |g(p)|\,=\,1\}\ \subset\ M,
\end{equation}
and refer to the points of $\Sigma$ as \emph{non-degenerate singularities}. The connected components of
$\Sigma$ are called \textbf{singular components}, and if a singular component is a closed curve in $M$, we call it a \textbf{singular loop}.    For an \emph{unbranched} maximal map (maxface) as in \cite{UMEHARA2006}, we use the following definition of completeness.

\begin{definition}[{\cite[Definition~4.1]{UMEHARA2006}}]\label{defn:complete}
A maximal map $X\,:\, M\,\longrightarrow\, \mathbb{E}_1^3$ is called \emph{complete} if there exist a compact
subset $C \,\subset\, M$ and a symmetric $2$-tensor $T$ on $M$ such that $T\,\equiv\, 0$ on $M\setminus C$
and $ds^2+T$ is a complete Riemannian metric on $M$, where $ds^2$ is the (possibly degenerate) induced
metric in \eqref{ds}.
\end{definition}

Umehara and Yamada \cite{UMEHARA2006} also give a  criterion for checking completeness:

\begin{prop}[{\cite[Corollary~4.8]{UMEHARA2006}}]\label{defn:complete-criteria}
The following three conditions are sufficient to ensure that $X\,:\, M \,\longrightarrow\, \mathbb{E}_1^3$
is complete in the sense of Definition~\ref{defn:complete}:
\begin{enumerate}
\item $M$ is biholomorphic to $\overline{M}\setminus\{p_1,\,\cdots,\,p_n\}$ for some compact Riemann surface
$\overline{M}$ (the \emph{compactification} of $M$) and distinct points $p_1,\,\cdots,\,p_n\,\in\,\overline{M}$.
\item $|g(p_i)|\,\neq\, 1$ at each puncture $p_i$.
\item The induced metric $ds^2$ given by \eqref{ds} is complete at each puncture $p_i$.
\end{enumerate}
\end{prop}

For each puncture $p_i\,\in\, \overline{M}$ (equivalently, each end of $M$), a neighborhood of $p_i$
in $M$ is mapped by $X$ to an \emph{end} of the maximal surface $X(M)$. Furthermore, if an end $p_i$ is
complete (condition (3) above), then the Weierstrass data $(g,\,\omega)$ extend meromorphically to $p_i$. In
that case, near $p_i$ one can take a local holomorphic coordinate $z$ centered at $p_i$ such that 
\[ 
(|g|^{-1}-|g|)^2\,|\omega|^2\ \;=\ \; \frac{1}{|z|^{2m}} +\ \text{(higher order terms)}
\] 
for some integer ${m\,\ge\, 1}$. (In particular, $m$ is the order of the pole of $ds^2$ at
the end $p_i$.)  Our maximal map will have simple ends.

Following Imaizumi--Kato \cite{IK}, an end $p_i$ is said to be \textit{simple} if the
two conditions below hold:
\begin{enumerate}
\item The Weierstrass data $(g,\,\omega)$ extend meromorphically across $p_i$, and 
\item poles of $g^{-1}\omega$,\, $g\,\omega$ and $\omega$ at $p_i$ are of order at most $2$. 
\end{enumerate}

In the next section, we turn to the problem of constructing examples of maximal maps with an arbitrarily large 
number of singular components.

\section{Complete maximal map of any genus and prescribed Gauss map}\label{sec:Mainsec}

Constructing a maximal map of genus \(p\) involves first constructing a genus-zero maximal map. The 
plan is to start with a genus-zero maximal map and then add handles to it. Below, we begin with the genus zero 
case as discussed by the third author with S. Mohanty in \cite{SaiPradip0genus}.

Let \( g \) be a meromorphic function on \( \mathbb{C} \cup \{\infty\} \) such that
\begin{equation}\label{gaussmap1}
(g)_{\infty}\ =\ \sum_{i=1}^{n} x_i\; p_i,\ \quad x_i \,\in\, \mathbb{N}^*, \quad 1 \,\leq\, i \,\leq\, n.
\end{equation}
Without loss of any generality, we assume that \( p_n \,=\, \infty \). Define \[ M_{0,n} \ :=\ \mathbb{C} \cup
\{\infty\} \setminus \{p_i\}_{i=1}^n .\] We start by outlining the setup.

\subsection{Required maximal map}\label{required maximal map setup}

For a given integer \( p \,\geq\, 1 \), the goal is to add handles to \( M_{0,n} \) to obtain a
genus \( p \) Riemann surface \( M_{p,2n} \) with \(2n\) punctures, and then construct a suitable holomorphic
one-form \(\Omega\) and a meromorphic function \(G\) on \(M_{p,2n}\) such that the following three statements
hold:
\begin{enumerate} 
\item The triple \((M_{p, 2n},\, G, \,\Omega)\) serves as the Weierstrass data for a complete maximal map
of genus $p$.

\item The singular components are at least \(2N\). Here, \(N\) is the cardinality
of the connected components of the subset \(\{z\,\in\, \mathbb{C} \cup \{\infty\} \,\,\big\vert\,\,\,
 |g(z)| \,=\, 1\}\).

\item Ensure that \(2N \,>\, p+1\) for a given \(p\).
\end{enumerate}

\subsection{Genus zero complete maximal maps}\label{subsec: genus zero}

Let \( a \,=\, (a_i) \,\in\, \mathbb{C}^{n-1} \) and \( b \,=\, (b_j) \,\in\, \mathbb{C}^{2n-1} \) be the
variables. Define
\begin{equation}\label{eq 4.1:initial f}
 f(z)\ =\ \sum_{i=1}^{n-1} \frac{a_i}{(z - p_i)^2} + \sum_{j=0}^{2n-2} b_j z^j.
\end{equation}

We want \(\left(g,\, \; \omega \,=\, g f(z) \, dz \right)\) to serve as the data for a maximal map on
\(M_{0,n}\). To achieve this, it is needed to solve for \((a,\, b)\) the following equations. For each
\(\gamma_i\), a loop surrounding \(p_i\), the following are required:
\begin{align*}
&\int_{\gamma_i} g \omega \,=\, \int_{\gamma_i} g^2 f(z) \, dz \,=\, 0,\\  
&\int_{\gamma_i} g^{-1} \omega \,=\, \int_{\gamma_i} f(z) \, dz \,=\, 0,\\ 
&\int_{\gamma_i} \omega \,=\, \int_{\gamma_i} g f(z) \, dz \,=\, 0.
\end{align*}

Since \(\int_{\gamma_i} g^{-1} \omega \,=\, 0\) for every \((a,\, b)\), we get  the following system of equations
in the variables \((a,\, b)\):
\begin{align}
  \label{eq1:0 genus period problem}  
  \operatorname{Res}_{p_i}(g f \, dz) &\,=\, 0,\\
  \label{eqn1:0 genus Period Problem secton}  
  \operatorname{Res}_{p_i}(g^2 f \, dz) &\,=\, 0.
\end{align}
This system will be referred to as System \(A\). Observe that this System A has \(2n - 2\) equations and \(3n - 
2\) variables. Denote the solution space of \(A\) by \(\text{Sol}_{(a,b)}\), which has dimension at least 
\(n\). Therefore, there is a non-trivial solution of the system of linear equations \(A\), and hence, a maximal 
map exists on \(M_{0,n}\). Moreover, since the dimension of the solution space is at least \(n\), and at least one 
\(b_j \,\neq\, 0\), it follows that this maximal map is complete at \(p_n\,=\,\infty\).

To get the maximal map, which is complete at other ends \(p_i\), $i\,\neq\, n$, too,  we replace \(f\) with \(f + \sum_{i=1}^{n-1} 
f_i\), where
\begin{equation}\label{equation:completeness argument}
f_i(z)\ =\ \sum_{k=2}^{2n+2} \frac{c_k^i}{(z - p_i)^k},\ \quad{c_k^i \,\in\, \mathbb{C}}.
\end{equation}

To achieve completeness at the ends, at least one non-zero coefficient \(c_k^i\) is required. Thus, to obtain a
complete maximal map on \(M_{0,n}\) with Weierstrass data \(\left(g, \; \omega \,=\,
g\big(f(z) + \sum_{i=1}^{n-1} f_i(z)\big)\, dz\right)\), it is needed to find a non-trivial solution
of the following linear
equations for each \(f_i\) (in addition to \eqref{eq1:0 genus period problem} and
\eqref{eqn1:0 genus Period Problem secton}):
\begin{align}
\label{eq 4.5:completeness 1}
\operatorname{Res}_{p_j}(g f_i \, dz) &\ =\ 0,\\
\label{eq 4.6:completeness 2}
\operatorname{Res}_{p_j}(g^2 f_i \, dz) &\ =\ 0,
\end{align}
where $j\,=\,1,\,2,\,\cdots,\,n-1$. Denote this system of equations by \(B_i\) for every
\(i \,=\, 1,\, 2,\, \cdots,\, n-1\).

We now have a homogeneous system of linear equations \(A\),\, \(B_1,\, \cdots,\, B_{n-1}\) with the variables
$$\{a \,=\, (a_i)_{1 \leq i \leq n-1},\ \; b \,=\, (b_j)_{0 \leq j \leq 2n-2},\ \; c \,=
\,(c_k^i)_{2 \leq k \leq 2n+2, \; 1 \leq i \leq n-1}\}.$$
The system \(B_i\) has \(2n+1\) variables and \(2n-2\) equations.
Denote the space of solutions of \(B_i\) by \(\text{Sol}_{c_i}\). The dimension of $\text{Sol}_{c_i}$ is at least $3$. Therefore, for each \(i\), at least one of the \(c_k^i\) is non-zero.

For each \(s_{(a,b)} \,\in\, \text{Sol}_{(a,b)}\) and \(s_{c_i} \,\in\, \text{Sol}_{c_i}\), there exists a complete 
maximal map of genus zero with Weierstrass data
$$(M_{0,n} =\ \mathbb{C} \cup
\{\infty\} \setminus \{p_i\}_{i=1}^n,\,\, \; g,\,\, \;
g (f(z) + \sum_{i=1}^{n-1} f_i(z)) \, dz).$$

Denote the space $\text{Sol}_{(a,b)}\times\prod_{i=1}^{n-1}\text{Sol}_{c^i}$ by $\text{Sol}_{(a,b,c)}$. Since
$\text{Sol}_{(a,b)}$ is of at least dimension $n$, and the dimension of each $\text{Sol}_{c^i}$ is at least $3$,
it follows that the solution space $\text{Sol}_{(a,b,c)}$ is a complex linear space of dimension at least $4n-3.$ 

Given that we have already obtained a genus \(0\) complete maximal map, the next natural step is to explore 
whether this map can be extended to a genus \(p\) complete maximal map, while ensuring that, when ``restricted'' to 
\(M_{0,n}\), the original ``maximal map is preserved'' (we shall explain this). The remaining part of this section 
focuses on addressing this question.

\subsection{Adding handles}\label{subsec:Adding handels}
We choose a disk $B$ on $\mathbb{C}\cup\{\infty\}$ such that
it contains both \( g^{-1}(S^1) \) and \(\{p_i\,\,\big\vert\,\,\, i\,=\,1, \,\cdots,\, n\}\).

Choose a set of distinct complex numbers \( z_1,\, \cdots,\, z_{2p+2} \) from \( \mathbb{C} \cup \{\infty\}
\setminus B \), such that the line segments \( C_i\, =\, \overline{z_i z_{i+1}} \), where
\( i\,=\,1,\, \cdots,\, 2p+1 \), lie outside \( B \) and do not intersect each other. Next, we cut slits along the
segments \( C_{2j+1} \) for \( j\,=\,0,\, \cdots,\, p \) and construct a hyperelliptic cover of
\( \mathbb{C} \cup \{\infty\} \), branched at the marked points \( z_1,\, \cdots,\, z_{2p+2} \). This process
results in a compact Riemann surface \( M_p \) of genus \( p \). There is a natural projection map
\begin{equation}\label{epi}
\pi\ :\ M_p \ \longrightarrow\ \mathbb{C} \cup \{\infty\}.
\end{equation}

Further, let \( E \) denote the preimage, under the map $\pi$ in \eqref{epi}, of the points
\( \{p_i \,\,\big\vert\,\, i\,=\,1,\, \cdots,\, n\} \). The cardinality of \( E \) is \( 2n \). Denote
\( M_{p, 2n} \ :=\ M_p \setminus E \).

Recall that for any \(0\, \neq\, \alpha\, \in\, \text{Sol}_{(a,b,c)}\), 
$$\left(M_{0,n} =\ \mathbb{C} \cup
\{\infty\} \setminus \{p_i\}_{i=1}^n,\,\;\; g,\, \;\;\; \omega
\,:=\, g(f(z) + \sum_{i=1}^{n-1} f_i(z)) \, dz\right)$$
is the Weierstrass data. We have three $1$-forms:
\begin{align}
  &\omega_1 \,=\, g \omega \,=\, g(z)^2 \left(f(z) + \sum_i f_i(z)\right) \, dz, \label{en:starting equation for omega1} \\
  &\omega_2 \,=\, g^{-1} \omega \,=\, \left(f(z) + \sum_i f_i(z)\right) \, dz, \label{en:starting equation for omega2} \\
  &\omega_3 \,=\, \omega \,=\, g(z) \left(f(z) + \sum_i f_i(z)\right) \, dz. \label{en:starting equation for omega3}
\end{align}
The pullback of these three 1-forms yield three holomorphic forms on \(M_{p, 2n}\), which we
denote by \(\widetilde{\omega}_1\),\, \(\widetilde{\omega}_2\) and \(\widetilde{\omega}_3\) respectively.
Note that $\widetilde{\omega_i}$  have poles only at \(E\,\subset\, M_{p}\).
Moreover these are defined for each $0\, \neq\, \alpha \,\in\, \text{Sol}_{(a,b,c)}$. 

We claim that there exists a non-trivial subspace \[V\ \subset\ \text{Sol}_{(a,b,c)}\] such that for each
\(0\, \neq\, \alpha \,\in\, \text{Sol}_{(a,b,c)}\), the data
\begin{equation}\label{eg}
\left(M_{p, 2n}, \;\; G \,:=\, \frac{\widetilde{\omega}_3}{\widetilde{\omega}_2}, \;\; \Omega
\,:=\, \widetilde{\omega}_3\right)
\end{equation}
define a maximal map as described in Section \ref{required maximal map setup}. This 
claim will be proved in steps by solving the period problem.

\subsubsection{Period Problem}\label{subsub:Period}

We begin with the basis of \(H_1(M_{p,2n},\, \mathbb{Z})\), which contains the following loops:

\begin{enumerate}
\item For each \(i\), the loops \(\Gamma_i^{1}\) and \(\Gamma_i^{2}\) are the lifts of the loop \(\gamma_i\)
in \(B \subset M_{0,n}\) that enclose \(p_i\) and no other \(p_j\) for \(j \,\neq\, i\). The loops
\(\Gamma_i^{1}\) and \(\Gamma_i^{2}\) enclose the inverse images of \(p_i\). 

\item The remaining \(2p\) loops are as follows: Choose \(\beta_i\), where \(i\, =\, 1,\, \cdots, \,p\), on 
\(M_{0,n}\) such that each \(\beta_i\) is single-sheeted and encloses only \(C_{2i+1}\) in its interior. Let 
\(\delta_j\), where \(j\,=\,1,\,\cdots,\, p\), be a two-sheeted loop that starts at \(C_{2j}\), ends at
\(C_{2j+1}\) on  one sheet, then starts again at \(C_{2j+1}\) and returns to \(C_{2j}\) on the other sheet.
\end{enumerate} 

We wish to establish the existence of a subspace \(V\ \subset\ \text{Sol}_{(a,b,c)}\) such that for each non-zero
\(\alpha \,\in\, V\), the pair \((G,\, \Omega)\) (see \eqref{eg}) solves the period conditions. For this we need
to solve the following equations. For all basis \(\gamma\,\in\, H^1(M_{p,2n},\,\mathbb Z)\):

\begin{align}
\int_\gamma \widetilde{\omega}_1 + \overline{\int_\gamma \widetilde{\omega}_2} &\ =\ 0, \\
\text{Re} \int_\gamma \widetilde{\omega}_3 &\ =\ 0.
\end{align}

By the definition of $\text{Sol}_{(a,b,c)}$,  for each $\alpha\,\in\, \text{Sol}_{(a,b,c)}$, we have 
\[\int_{\Gamma_i^{1,2}}\widetilde{\omega}_j=0.\]

Moreover, by construction, the forms $\omega_1,\,\omega_2,\,\omega_3$ are holomorphic (and exact)  on
$\mathbb{C}\cup\{\infty\}\setminus B$. Therefore, for single-sheeted loops \(\beta_i\), we must have
$$\int_{\beta_i} \widetilde{\omega}_j \ =\ 0$$
for all \(i\,=\,1,\,\cdots,\,p\) and \(j\,=\,1,\,2,\,3\).

Thus, the remaining task is to find a non-trivial solution space \( V \,\subset\, \text{Sol}_{(a,b,c)} \) such 
that \( V \) is the kernel of the map
\[
F\ :\ \text{Sol}_{(a,b,c)} \ \longrightarrow\ \mathbb{C}^{2p}
\]
defined by
\[
\alpha\ \longmapsto\ \left(\int_{\delta_j}\widetilde{\omega}_1 + \overline{\int_{\delta_j}\widetilde{\omega}_2},\;\;\; \text{Re}\int_{\delta_j} \widetilde{\omega}_3\right)_{j=1,\ldots, p}.
\]

Since \(\widetilde{\omega}_i\) and \(\omega_i\) have no poles around each Weierstrass point or on the line
segments \(C_k 
\,=\, \overline{z_k z_{k+1}}\), we have
\[
\int_{\delta_j}\widetilde{\omega}_i\ =\ 2 \int_{z_{2j}}^{z_{2j+1}} \omega_i.
\]
Therefore, it follows that
\[
F(\alpha)\ =\ 2 \left(\int_{z_{2j}}^{z_{2j+1}} \omega_1 + \overline{\int_{z_{2j}}^{z_{2j+1}} \omega_2},
\;\;\; \text{Re} \int_{z_{2j}}^{z_{2j+1}} \omega_3 \right)_{j=1,\ldots, p}.
\]

Since \(F\) is a linear map from \(\text{Sol}_{(a,b,c)}\), if \(4n-3 \,>\, 2p\) (the dimension of 
\(\text{Sol}_{(a,b,c)}\) is at least \(4n-3\)), then there is a non-trivial solution \(\alpha \,\in\, V \,=\, 
\text{Ker}(F)\) for which the period problem is solved. The dimension of \(V\) is at least \(4n-2p-3\).

\subsubsection{Conformal Condition}\label{subsub:Conform}

Away from the Weierstrass points, in a coordinate representation, we have
\[
\widetilde{\omega}_1 \widetilde{\omega}_2 - \widetilde{\omega}_3^2\ =\ \omega_1 \omega_2 - \omega_3^2\ =\ 0.
\]
Since \(M_{p, 2n}\) is connected, it follows that \(\widetilde{\omega}_1 \widetilde{\omega}_2 - \widetilde{\omega}_3^2\, \equiv\, 0\).

Subsections \ref{subsub:Period} and \ref{subsub:Conform} show that the triple \[(M_{p, 2n},\, \; G \,:=\,
\frac{\widetilde{\omega}_3}{\widetilde{\omega}_2},\, \; \Omega \,:=\, \widetilde{\omega}_3)\] defines a maximal map. Now, all \(2n\) ends corresponding to the points in \(E\) are complete.
Indeed, this can be seen from the following two facts:

\begin{enumerate}
\item The existence of a non-zero solution of the system \(B_i\) ensures that the order of the pole at
each point of \(E\) is at least \(2\).

\item At all end points, the Gauss map \(G\) has poles, hence the singularity set is compact.
\end{enumerate}

Summarizing all the above, we have the following:

\begin{prop}\label{prop:exsistece of maximal map}
Given a meromorphic map \(g\,:\, \mathbb{C} \cup \{\infty\}\, \longrightarrow\, \mathbb{C}\) with \(n\) poles, 
if \( 4n-2p-3 \,> \, 0\), then there exists a complete maximal map of genus \(p\), defined on a hyperelliptic cover 
of \(\mathbb{C} \cup \{\infty\}\), such that it has \(2n\) complete ends and its Gauss map \(G\) is the lift of 
\(g\).
\end{prop}

We note that the branch points of $g$ lie outside the disk \( B \) while \( g^{-1}(S^1) \) lies inside the disk, and 
\( G \,=\, g \) when restricted to \( \pi^{-1}(B) \). Hence, if \( N \) is the cardinality of the connected 
components of the set \( \{z \,\,\big\vert\,\,\, |g(z)| \,=\, 1\} \), then the number of singular components of the 
maximal map as in Proposition \ref{prop:exsistece of maximal map} is at least \( 2N \).

Using the above, we will prove the following theorem.

\begin{theorem}\label{thm:genus,singularity}
For given \( p \) and \( N \,>\, p + 2 \), there exists a complete maximal map of genus \( p \) with \( 2N \) complete ends and having at least \( 2N \) singular components.
\end{theorem}

\begin{proof}
We start with \( N \) and \( p \) as given. We find a meromorphic function \( g \) on \(\mathbb{C} \cup \{\infty\}\) such that it
has \( N \) poles and \(\{z \,\,\big\vert\,\, g'(z)\, =\, 0\} \cap g^{-1}(S^1) \,=\, \emptyset\).

Since \( g \) has \( N \) distinct poles, the degree of \( g \) is at least \( N \). Furthermore, since $g^\prime(z)\,\neq\, 0$ on the
set $g^{-1}(S^1)$, it follows that  \( g^{-1}(S^1) \) has at least \( N \) connected components.

Since \( N-p-2\,>\,0\), we have $4N-2p-3\,>\,0$. Therefore, with $g$ as in the previous paragraph, we can construct a maximal map as
described in Proposition \ref{prop:exsistece of maximal map}. This maximal map will be complete and will have at least \( 2N \)
singular components.
\end{proof}

Following are a few remarks on the maximal maps described in Theorem \ref{thm:genus,singularity}.

\begin{remark}\mbox{}
\begin{enumerate}
\item The maximal map will have branched singularities at the locations where \(\Omega\) is zero.

\item The maximal map  will have \(2N\) complete ends. Therefore, our construction increases the singularity set as the number of complete ends increases.

\item Since the solution space \(V\) for our maximal map is of dimension at least \(4n-2p-3\), there is a large class of such maximal maps. While not all these maximal maps will necessarily be maxfaces, there is a possibility of deforming the Gauss map of these maximal maps to achieve a maxface.
\end{enumerate}
\end{remark}

The above construction leaves open a significant freedom to add handles or create embedded ends. Our next goal 
is to construct a maximal map of higher genus with an arbitrary number of simple ends and a large number of 
singular components. This involves deforming the earlier data and adding simple ends rather than creating 
holes (i.e., adding embedded ends).

\subsection{Higher genus maximal maps with an arbitrary number of simple ends and a prescribed number of 
singular components.}\label{subsec:Simpleends}

Let \( g \) be a meromorphic function as defined in \eqref{gaussmap1}. Recall that \( B \) is the disk in $\mathbb{C}\cup\{\infty\}$ and it contains the set \( \{ p_i\,\,\big\vert\,\, i \,=\, 1, \,\cdots,\, n \} \). We choose \( q_j
\,\in\, B \), where \( 1 \,\leq\, j \,\leq\, m \), that are distinct from \(\{ p_i\}_{i=1}^n \)s.

We aim to find the Weierstrass data \((G, \,\Omega)\) for the complete maximal map on \(\mathbb{C} \cup \{\infty\} \setminus 
\pi^{-1}\left(\{p_1,\, \cdots,\, p_n,\, q_1,\, \cdots,\, q_m\}\right)\), where the points of \(E\, =\, \pi^{-1}(\{p_1,\, \cdots,\, p_n\})\)
are complete ends 
and the points of \(F\,:=\, \pi^{-1}(\{q_1,\,\cdots,\, q_m\})\) are simple (in particular, embedded) ends. This works as in the proof
of Proposition \ref{prop:exsistece of maximal map} with only one difference.

The only difference is that instead of the expression \(f(z) + \sum_i f_i(z)\), we
need to take \[\frac{f(z) + \sum_i f_i(z)}{\prod_j (z - 
q_j)^2}.\]

With these new expressions and by employing a similar technique as discussed in the proof of Proposition 
\ref{prop:exsistece of maximal map}, we can reach the conclusion that if \(4n - 2p - m - 3 \,>\, 0\), then there exists a 
maximal map of genus \(p\) with \(2n\) complete ends and \(2m\) simple ends, and with singular components at 
least twice the number of the connected components of the set \(\{z\,\,\big\vert\,\,\, |z| \,=\, 1\}\).

Furthermore, similar to Theorem \ref{thm:genus,singularity}, we have the following theorem:

\begin{theorem}\label{thm:simple and Complete ends}
Given \(p\), \(m\), and \(N\) such that \(N - \lceil \frac{m}{2} \rceil \,>\, p + 2\), there exists a complete maximal map of
genus \(p\) with \(2N\) complete ends and \(2m\) simple ends, and with at least \(2N\) singular components.
\end{theorem}
\section{Concluding Remarks}\label{sec:conclusion}

In the previous section, we constructed a family of complete maximal maps in the Lorentz-Minkowski space $\mathbb{E}^3_1$, of arbitrary genus \(p\), whose nondegenerate singular sets consist of an arbitrarily large number of connected components. The main results, Theorems~\ref{thm:genus,singularity} and \ref{thm:simple and Complete ends}, establish that for any given genus \(p\) and any sufficiently large integer \(N\), there exists a complete maximal map of genus \(p\) with at least \(2N\) connected singular loops, along with a large number of ends, including both complete and simple ends. However, for a given maximal map in this family, while we have a guaranteed lower bound on the number of nondegenerate singular components, a precise lower count remains unknown.

Empirically, in every known example of a maxface in the literature, one observes that a singular loop appears in each handle, and the immersion fails around the handle region. Figure~\ref{fig:singularity aroundneck} illustrates this phenomenon: part (A) shows a segment of the toroidal maxface of Kim–Yang \cite{Kim2006}, while part (B) depicts a portion of a maxface from the Lorentzian Costa–Hoffman–Meeks (CHM) family \cite{chen2024singularitiesmaxfacesconstructednodeopening}. In both cases, the singularities wrap around the necks of the surface, indicating a breakdown of immersion in these regions.

\begin{figure}[H]
    \centering
      \begin{subfigure}[b]{0.35\textwidth}
        \includegraphics[scale=0.32]{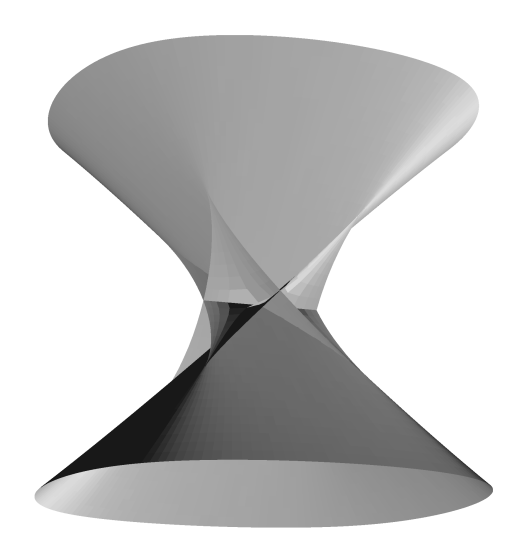}
        \caption{Kim–Yang's toroidal maxface \cite{Kim2006}}
    \end{subfigure}\;\;\;\;\;\;\;\;\;\;\;\;\;\;\;
    \begin{subfigure}[b]{0.35\textwidth}
        \includegraphics[scale=0.4]{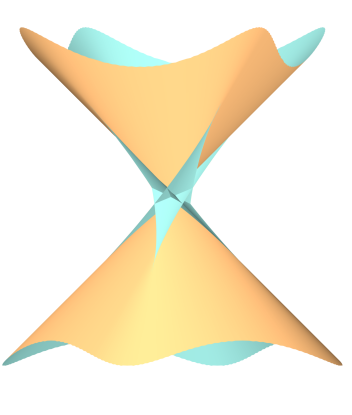}
        \caption{Lorentzian CHM maxface (partial) \cite{chen2024singularitiesmaxfacesconstructednodeopening}}
    \end{subfigure}
    \caption{Singularities wrapping around neck regions}
    \label{fig:singularity aroundneck}
\end{figure}

\begin{figure}[H]
    \centering
    \includegraphics[scale=0.4]{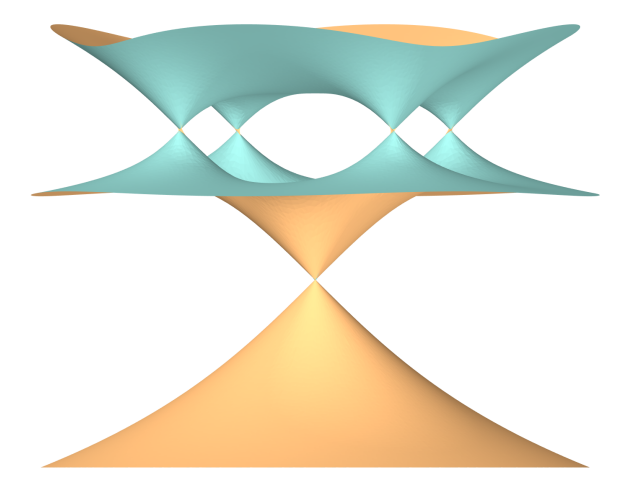}
    \caption{Lorentzian Costa–Hoffman–Meeks maxface \cite{chen2024singularitiesmaxfacesconstructednodeopening}}
    \label{Fig:Full CHM surface}
\end{figure}

This behavior is persistent across all presently known constructions. It suggests a possible dependence of the number of singular components on the genus of the surface. In view of our constructions and empirical observations summarized in Table~\ref{Table of singular components}, it is plausible that the number of connected singular components might be bounded below by \(p+1\). 
\section{Acknowledgment}
I. Biswas is partially supported by a J. C. Bose Fellowship (JBR/2023/000003).\;\;S. Paul is supported by UGC JRF (Beneficiary Code: BININ01854128).

\bibliography{biblio.bib}
\end{document}